\documentclass{raexmod}
\usepackage{amsmath, amssymb, amsthm, amsfonts}

\newcommand{\p}{\partial}

\newcommand{\dd}{{\rm d}}

\begin{Author}
	\FirstName{Ettore}\LastName{Minguzzi}
	\PostalAddress{Dipartimento di Matematica e Informatica ``U. Dini'', \\
 Universit\`a
degli Studi di Firenze,\\
 Via S. Marta 3,  I-50139 Firenze, Italy}
	\Email{ettore.minguzzi@unifi.it }
	\Thanks{Research partially supported by GNFM of INDAM}
\end{Author}

\begin{MathReviews}
	\primary{ 26B05}
	\primary{26B30}
	\secondary{26A16}
\secondary{97I40}
\secondary{83C99}
\end{MathReviews}

\begin{KeyWords}
\end{KeyWords}

\title{The equality of mixed partial derivatives under weak differentiability conditions}

 \newtheorem{theorem}{Theorem}
 \newtheorem{lemma}[theorem]{Lemma}

 \theoremstyle{definition}
\newtheorem{definition}[theorem]{Definition}

 \theoremstyle{remark}
\newtheorem{remark}[theorem]{Remark}

\begin{document}

\maketitle

\begin{abstract}
We review and develop two little known results on the equality of mixed partial derivatives which can be considered the best results so far available in their respective domains. The former, due to Mikusi\'nski and his school, deals with equality at a given point, while the latter, due to Tolstov, concerns equality almost everywhere.
Applications to differential geometry and General Relativity are commented.
\end{abstract}

\section{Introduction}

As it is well known under reasonable conditions the mixed partial
derivatives of a real function coincide. This result  has a long
history and several distinguished scholars provided proofs including
Euler and Clairaut. However, according to Lindel\"of none of those
proofs was free of errors or tacit assumptions, so that historians
give credit for the first correct proof  to H. A. Schwarz, see
\cite{higgins40} for a  nice historical account.


Actually, the first correct proof of the equality of mixed partial
derivatives was obtained by Cauchy who improved and amended a
previous proof by Lagrange. However, they assumed the existence and
continuity of the derivatives $\p_1^2 f$, $\p_2^2 f$. Schwarz
removed this assumption and showed also that the continuity of
$\p_1\p_2 f$ could be obtained from the other hypothesis. Let
$O=(a,b)\times (c,d)\subset \mathbb{R}^2$. He proved
\cite{schwarz73}:
\begin{itemize}
\item[1.]  Let $f \in C^1(O,\mathbb{R})$ and suppose that $\p_2 \p_1 f$
exists and belongs to $C(O,\mathbb{R})$. Then $\p_1 \p_2 f$ exists
and $\p_1 \p_2 f=\p_2 \p_1 f$.
\end{itemize}

It is natural to ask whether the assumptions can be weakened.
Stronger versions can be found in the first published studies of
this problem. For instance, Dini in his ``Lezioni di Analisi
Infinitesimale'' \cite[p.\ 164]{dini07} does not assume $f\in
C^1(O,\mathbb{R})$, but demands just the existence of the partial
derivatives and the continuity of $\p_2 f$ in $y$.\footnote{In
Dini's book ${\p^2 f}/{\p x\p y}$ means $\p_2\p_1 f$.
We stress that in the
terminology of this article a necessary condition for a limit, such
as a partial derivative, to exist will be its finiteness. That is,
we do not tacitly use the extended real line, as some other authors
do. This convention allows us to write just ``exists'' in place of
``exists and is finite''.}

The strongest result in this direction seems to have been obtained
by Peano who removed the assumption on the continuity of $\p_2 f$
from Dini's version
 \cite{peano90}. Peano's version can be found in Rudin
 \cite{rudin76}.
\begin{itemize}
\item[2.]  Let $f\colon O\to \mathbb{R}$. Suppose that $\p_1 f$, $\p_2 f$ and  $\p_2\p_1 f$
exist on  $O$ and that the latter is continuous at $(x_0,y_0)$. Then
$\p_1\p_2 f(x_0,y_0)=\p_2\p_1 f(x_0,y_0)$.
\end{itemize}
The continuity of  $f$ is not assumed and there are indeed
discontinuous functions which admit everywhere partial derivatives
at any order \cite{kimura60}.

 Many authors tried to weaken
the conditions of Schwarz's theorem in other directions. Young
proved the following result (see Apostol \cite[Theor.\
12.12]{apostol74}):
\begin{itemize}
\item[3.] Let $f\colon O\to \mathbb{R}$. If both
partial derivatives $\p_1 f$ and $\p_2 f$ exist in neighborhood of
$(x_0,y_0)\in O$ and if both are differentiable at $(x_0,y_0)$, then
$\p_2 \p_1 f(x_0,y_0)=\p_1\p_2 f(x_0,y_0)$.
\end{itemize}
We observe that these assumptions imply that $\p_1 f$ and $\p_2 f$
being  continuous at $(x_0,y_0)$ are bounded in a neighborhood of
this point, thus $f$ is Lipschitz and hence continuous in such
neighborhood. As with Lagrange-Cauchy version, Young's result
assumes the existence of both $\p_1^2 f(x_0,y_0)$ and $\p_2^2
f(x_0,y_0)$.


We are now going to prove a results which improves Peano's. It is
based on the concept of strong differentiation also introduced by
himself in \cite{peano92,dolecki12}. It seems that he did not
realize the usefulness of strong differentiation for the problem of
the equality of mixed derivatives, possibly because he investigated
the latter problem before the introduction of this derivative.

\begin{remark}
Recently the notion of strong differentiation has  received renewed
attention since it has been proved that the exponential map of
Lipschitz connections or sprays over $C^{2,1}$ manifolds is strongly
differentiable at the origin \cite{minguzzi13d}. This fact implies
that the exponential map is a Lipeomorphism near the origin. Thus
this notion proves important to do differential geometry under weak
differentiability conditions.
\end{remark}

\begin{definition}
A function $f\colon B \to \mathbb{R}^c$, $B\subset \mathbb{R}^a
\times \mathbb{R}^b$, $(x,z) \mapsto f(x,z)$, is said to be {\em
partially strongly differentiable} with respect to $x$ at
$(x_0,z_0)\in \bar{B}$, with differential $\p_1 f(x_0,z_0)$ if for
every $\epsilon>0$, there is a $\delta>0$ such that for every
$x_1,x_2, z$ such that $\Vert x_1-x_0\Vert <\delta$, $\Vert
x_2-x_0\Vert <\delta$, $\Vert z-z_0\Vert <\delta$, $(x_1,z)\in B$,
$(x_2,z)\in B$,
\[
\Vert f(x_2,z)-f(x_1,z)-\p_1 f(x_0,y_0) (x_2-x_1)\Vert \le \epsilon
\Vert x_2-x_1\Vert.
\]
If $f$ does not depend on $z$ then  $\p_1 f$ is called  {\em
differential} and is denoted $d f$.
\end{definition}

Clearly, if a function is strongly differentiable then it is
differentiable, and the strong differential coincides with the
differential. Some interesting properties are
\cite{peano92,esser64,nijenhuis74}:
\begin{itemize}
\item[(i)] If $f$ is strongly differentiable at $p$ then its satisfies a
Lipschitz condition in a neighborhood of $p$.
\item[(ii)] If $f$ is differentiable in a neighborhood of $p$ and the
differential is continuous at $p$ then it is strongly differentiable
at $p$. Conversely, if $f$ is strongly differentiable at $p$ and the
differential exists in a neighborhood of $p$ then  the differential
 is continuous at $p.$ A similar version for partial strong
differentiation holds (this point is an easy consequence of the mean
value theorem).
\item[(iii)] If $f$ is strongly differentiable over a subset $A\subset E$ then
the strong differential is continuous over $A$ with respect to the
induced topology.
\end{itemize}
In particular, a function is strongly differentiable in an open set
$O$ if and only if it is continuously differentiable on $O$.

The concept of strong differentiation has some advantages over that
of ordinary (Frechet) differentiation. In particular, it allows us
to obtain simpler and stronger results through shorter proofs. It
serves better the intuition and at the elementary level could
possibly
replace the usual differentiation in
elementary textbooks on analysis. Indeed, it extends the the range
of applicability of some key results in analysis by removing some
continuity assumptions on derivatives.
 For instance:
\begin{itemize}
\item[(iv)] A function which is partially strongly differentiable with
respect to all its variables at a point $p$ is also totally strongly
differentiable at that point  $p$ \cite{nijenhuis74}.
\item[(v)] If a function $f\colon \mathbb{R}\to \mathbb{R}$ has positive strong
derivative at a point then it is increasing in a neighborhood of
that point. More generally, if a function $f:\mathbb{R}^n\to
\mathbb{R}^n$ has invertible strong differential at a point $p$ then
it is injective in a neighborhood of $p$ and the inverse is strongly
differentiable at $f(p)$ with $d f^{-1} (f(p))=(df(p))^{-1}$
(Leach's inverse function theorem \cite{leach61})
\end{itemize}

Let us observe that the usual assumptions that make the
corresponding results hold for  ordinary differentiation imply that
the differential is continuous in a neighborhood of $p$. As observed
above these assumptions serve essentially to assume strong
differentiability in a neighborhood without naming it. The point of
using strong differentiability is that strong differentiability at a
point suffices.

Peano's theorem on the equality of mixed partial derivatives at
$(x_0,y_0)$ demands the existence of $\p_2\p_1 f$ in a neighborhood
of $(x_0,y_0)$ and its continuity there. By (ii) above, $\p_1 f$ is
partially strongly differentiable with respect to $y$ at
$(x_0,y_0)$, thus one can ask whether the previous conditions can be
replaced by partial strong differentiability. The answer is
affirmative. The author reobtained the next theorem unaware of a previous result by Mikusi\'nski \cite{mikusinski72,mikusinski73,mikusinski78} subsequently generalized to Banach spaces by Sk\'ornik \cite{skornik83}. Mikusi\'nski calls  ``full derivative" the Peano's strong derivative  and does not give references so that it was quite hard to spot his important work \cite{mikusinski78}.

He reobtains first results due to Peano and more advanced results such as Leach's inverse function theorem, but he also obtain new results on the role of strong differentiation in integration theory. In \cite[Chap.\ 12]{mikusinski78} he provides the best and most complete introduction to strong derivatives up to date. The fact that he published those results in Polish \cite{mikusinski73} and in some sections in a book devoted to quite different problems did not help to spread knowledge of his important contributions. For completeness we include the next proof as it is different from Mikuskinski's and has weaker assumptions.


\begin{theorem}
Let $f\colon O\to \mathbb{R}$. Suppose that the partial derivative
$\p_1 f$ exists on  $O$ and that it is partially strongly
differentiable with respect to $y$ at $(x_0,y_0)$. Then, denoting
with $A\subset O$ the subset where $\p_2 f$ exists, provided
$(x_0,y_0)\in \bar{A}$,  $\p_2 f(:=\p_2 f\vert_A)$ is partially
strongly differentiable with respect to $x$ at $(x_0,y_0)$ and
$\p_1\p_2 f(x_0,y_0)=\p_2\p_1 f(x_0,y_0)$.
\end{theorem}

We stress that while the assumptions are weaker than Peano's, the
conclusion is stronger.
For instance, the previous theorem implies that
if $ \p_1\p_2 f(\cdot,y_0)$ exists in a neighborhood of $x_0$ then
it is continuous at $x_0$.

  Esser and Shisha \cite{esser64,nijenhuis74} construct a simple
function $h(y)$ defined on an open set of $y=0$ which is not
everywhere differentiable on any neighborhood of 0 but which is
strongly differentiable at $0$. Then $f(x,y)=x h(y)$ satisfies the
assumptions of our theorem but not those of Peano's.

\begin{proof}
Let $\epsilon >0$. Since $\p_1 f$ is partially strongly
differentiable with respect to $y$ at $(x_0,y_0)$ there is
$\delta(\epsilon)>0$ such that for every $\tilde x \in (a,b)$,
$\tilde y_1, \tilde y_2\in (c,d)$, $\vert \tilde x-x_0\vert
<\delta$, $\vert \tilde y_1-y_0\vert <\delta$, $\vert \tilde
y_2-y_0\vert <\delta$, we have
\begin{equation} \label{ied}
\vert \p_1 f(\tilde x,\tilde y_2)-\p_1 f(\tilde x,\tilde
y_1)-\p_2\p_1 f(x_0,y_0) (\tilde y_2-\tilde y_1)\vert \le \epsilon
\vert \tilde y_2-\tilde y_1\vert .
\end{equation}
Given $\epsilon>0$ let $\delta(\epsilon)>0$ be as above. Let
$x_1,x_2\in (a,b)$ be such that $\vert x_1-x_0\vert <\delta$, $\vert
x_2-x_0\vert <\delta$, and let $y\in (c,d)$ be such that $\vert
y-y_0\vert <\delta$. Furthermore, let them be such that $(x_1,y) \in
A$, $(x_2,y)\in A$.

Let $y_1,y_2\in (c,d)$, $y_1\ne y_2$, be arbitrary and such that
$\vert y_1-y_0\vert <\delta$, $\vert y_2-y_0\vert <\delta$. Let
$u(t):=f(t,y_2)-f(t,y_1)$, then by the existence of $\p_1 f$ and by
the mean value theorem there is $x\in (a,b)$, $\vert x-x_0\vert
<\delta$, such that
\[
\p_1 u(x) (x_2-x_1)=u(x_2)-u(x_1).
\]
Equation (\ref{ied}) holds for these values for $x,y_1,y_2$, thus
\begin{align*}
\vert f(x_2,y_2)-f(x_2,y_1)-f(x_1,y_2)+f(x_1,y_1)-&\,\p_2\p_1
f(x_0,y_0) (y_2-y_1) (x_2-x_1)\vert\\&\le \epsilon \vert
x_2-x_1\vert \, \vert y_2-y_1\vert.
\end{align*}
Dividing by $\vert y_2-y_1\vert$, setting $y_1=y$ and taking the
limit $y_2\to y$ we obtain
\[
\vert \p_2 f(x_2,y)-\p_2 f(x_1, y) -\p_2\p_1 f(x_0,y_0)
(x_2-x_1)\vert \le \epsilon \vert x_2-x_1\vert,
\]
which means that $\p_2 f$ has partial strong differential with
respect to $x$ at $(x_0,y_0)$ given by $\p_2\p_1 f(x_0,y_0)$, that
is $\p_1\p_2 f(x_0,y_0)=\p_2\p_1 f(x_0,y_0)$.
\end{proof}

The concept of strong differentiation lead us to a satisfactory
result which we can summarize as follows:

\begin{itemize}
\item[(vi)] If the strong derivative $\p_2\p_1 f(x_0,y_0)$ exists
and it makes sense to consider the strong derivative $\p_1\p_2
f(x_0,y_0)$ (that is, $\p_2 f$ exists in a set which accumulates at
$(x_0,y_0)$) then the latter exists and  they coincide.
\end{itemize}


\section{Equality almost everywhere}

We might also ask to what extent the equality of mixed partial
derivatives holds for functions which admit those second derivatives
almost everywhere. Some important results have been obtained for
convex functions. A well known result by Alexandrov establishes that
convex functions admit a generalized Peano derivative of order 2
in the sense that for almost every $x$
\[
f(x+h)=f(x)+L(h)+ A(h,h)+o_x(\vert h\vert^2),
\]
where $L$ is a linear map and $A$ is a quadratic form. Less clear is
whether $A$ can be obtained from the differentiation of the
generalized differential of $f$ and whether  such double
differentiation gives a symmetric Hessian. The affirmative answer to
this question has been established by Rockafellar
\cite{rockafellar99}.

The Russian mathematician G.\ P.\ Tolstov clarified several
questions related to the equality of mixed derivatives in two papers
published in 1949.
Unfortunately, only one of those articles was translated into
English \cite{tolstov49}, so that the interesting results  contained
in the other paper \cite{tolstov49b} have been largely overlooked by
the mathematical community. Most space in those papers is devoted to
the construction of counterexamples. In fact, he proved
\cite{tolstov49b}:
\begin{itemize}
\item[4.] There exists a function $f\in C^1(O,\mathbb{R})$, the mixed second derivatives of which
exist at every point of $O$ but such that $\p_2\p_1 f\ne \p_1\p_2 f$
on a set $P\subset O$ of positive measure.
\item[5.] There exists a function $f\in C^1(O,\mathbb{R})$, the mixed second derivatives of which
exist almost everywhere in $O$ and such that $\p_2\p_1 f\ne \p_1\p_2
f$ almost everywhere in $O$.
\end{itemize}
On the positive direction he improved Young's theorem as follows
\cite{tolstov49}:
\begin{itemize}
\item[6.] If the function $f$ has all second derivatives everywhere
in $O$, then the equality of mixed derivatives holds in $O$.
\end{itemize}
This theorem with {\em existence} replaced by {\em existence almost
everywhere} in both the hypothesis and thesis had been already
proved by Currier \cite{currier33}.

These type of results have still an undesirable feature, for they
place conditions on the existence of the  double derivatives $\p^2_1
f$ and $\p^2_2 f$. Actually, Tolstov  obtained some results which do
not place conditions on the homogeneous second derivatives. In the
remainder of this work we shall review and develop them. In
particular, we shall stress the importance for applications of the
Lipschitz conditions on the first derivatives.

We start with an important Lemma from Tolstov's paper. Since there
are no published English translations, we provide the proof.

\begin{lemma}[Tolstov \cite{tolstov49b}] \label{lem}
Let $O=(a,b)\times (c,d) \subset \mathbb{R}^2$ and let
\[
f(x,y)=\int_a^x \dd u\int_c^y h(u,v) \dd v.
\]
where $h \in L^1(\bar{O},\mathbb{R})$.  Then there is a measurable
set $e_1\subset (a,b)$ with $\vert e_1\vert=b-a$ such that for every
$x\in e_1$  and for every $y\in (c,d)$
\begin{equation} \label{jud}
\p_1 f(x,y)=\int_c^y h(x,v) \dd v.
\end{equation}
\end{lemma}

Clearly, by Fubini's theorem the integrals in the definition of
$f(x,y)$ can be exchanged  and hence a similar statement holds for
the derivative with respect to $y$. Fubini's theorem will play a
very important role in the proof of this lemma and in the proofs of
the next theorems. The reader is referred to Aksoy and Martelli
\cite{aksoy02} for a discussions of the relationship between
Fubini's and Schwarz's theorems.

\begin{proof}
 Differentiating $f(x,y)$ with respect to
$x$ we obtain for $x \in X_y\subset [a,b]$ with $\vert
X_y\vert=b-a$,
 (Fundamental Theorem of Calculus e.g.\ \cite[Theor.\ 8.17]{rudin70})
\begin{equation} \label{bos}
\p_1 f(x,y)=\int_c^y  h(x,v)\, \dd v .
\end{equation}

Let $\Lambda=\cup_y  ( X_y\times \{y\} )$, so that $\vert
\Lambda\vert=(b-a)(d-c)$. Equation (\ref{bos}) holds for $(x,y)\in
\Lambda$.  Let $Y_x\subset (c,d)$ be the coordinate slices defined
by $\{x\}\times Y_x=(\{x\}\times (c,d))\cap \Lambda$, or
equivalently $Y_x=\pi_2(\pi_1^{-1}(x)\cap \Lambda)$. Fubini's
theorem applied to the characteristic function of $\Lambda$ gives
that there is some $e_1 \subset (a,b)$, $\vert e_1\vert=b-a$, such
that for every $x\in e_1$, $\vert Y_x\vert= c-d$. Let $x\in e_1$,
for $y\in Y_x$ Eq.\ (\ref{bos}) is true. We wish to show that it
holds for any $y\in(c,d)$. Let $y\in (c,d)$  and let $h_n$ be any
sequence converging to zero. Let
\[
\varphi^\pm_n(y):=\frac{1}{h_n} \int_x^{x+h_n}
\dd u\int_c^y h^\pm(u,v) \dd v.
\]
where $h^+$ and $h^-$ are the positive and negative parts of $h$,
respectively.  Since the functions $\varphi^+_n(y)$ are monotone and
continuous and converge in a dense subset (for $n \to \infty$), to the continuous function
$\int_c^y h^+(x,v)\, \dd v$ they do the same everywhere on $(c,d)$,
and analogously  for $\varphi^-_n(y)$. By the arbitrariness of
$h_n$, Eq.\ (\ref{bos}) is true for every $y\in (c,d)$ provided
$x\in e_1$.
\end{proof}

\begin{remark}
Notice that the Fundamental Theorem of Calculus (e.g.\ \cite[Theor.\
8.17]{rudin70}) states that Eq.\ (\ref{jud}) is true for $x\in
e_1(y)$, where $\vert e_1(y) \vert=b-a$, but $e_1(y)$ might depend
on $y$. The previous lemma states that $e_1$ does not depend on $y$.
\end{remark}

A corollary is:

\begin{theorem}[Tolstov \cite{tolstov49b}] \label{the}
Let $h$, $f$ and $O$ be as in Lemma \ref{lem} above. There are
measurable sets $e_1$ and $e_2$ with $\vert e_1\vert=b-a$, $\vert
e_2\vert=d-c$, such that
\begin{itemize}
\item[(a1)] Anywhere in $\{x\}\times (c,d)$ with $x\in e_1$, we have
\begin{equation} \label{ni1}
\p_1 f(x,y)=\int_c^y h(x,v) \dd v.
\end{equation}
\item[(a2)] Anywhere in $(a,b)\times \{y\}$ with $y\in e_2$, we have
\begin{equation} \label{ni2}
\p_2 f(x,y)=\int_a^x h(u,y) \dd u.
\end{equation}
\item[(b)] There exist a measurable  set $E\subset e_1\times e_2$ with $\vert
E\vert=(b-a) (d-c)$, such that for every $(x,y)\in E$ the mixed
derivatives exist, and
\begin{equation} \label{jis}
\p_2\p_1 f(x,y)=h(x,y)=\p_1\p_2 f(x,y).
\end{equation}
Moreover, for every $x\in e_1$, $\vert \pi_1^{-1}(x)\cap E\vert=d-c$, and for every $y\in e_2$, $\vert \pi_2^{-1}(y)\cap E\vert =b-a$.
\end{itemize}
\end{theorem}

With respect to Tolstov's paper we have included the last statement of  point (b).
Although this inclusion lengthens the proof    we give this version in order to be as complete as possible. A similar statement could be included at the end of the next theorems.

\begin{proof}
Let $e_1$ and $e_2$ be as in Lemma \ref{lem}, then (a1) and (a2) are
a rephrasing of that lemma. Differentiating Eq.\ (\ref{ni1}) with
$x\in e_1$ with respect to $y$ we obtain that there is $e^2(x)\subset
(c,d)$, $\vert e^2(x)\vert=d-c$ such that for $y\in e^2(x)$ the derivative
$\p_2 \p_1 f(x,y)$ exists and
\begin{equation} \label{bjy}
\p_2 \p_1 f(x,y)=h(x,y).
\end{equation}
By taking the intersection of $e^2(x)$ with $e_2$, if necessary, we can assume that $e^2(x)\subset e_2$.

Let $E^1=\cup_{x\in e_1} \{x\}\times e^2(x)$, so that $\vert E^1\vert=(b-a)(d-c)$ and on $E^1$ Eq.\ (\ref{bjy}) holds true. Observe that $\pi_1(E^1)\subset e_1$ and $\pi_2(E^1)\subset e_2$. By Fubini's theorem there is $c_2 \subset e_2$, $\vert c_2\vert = d-c$, such that for every $y\in c_2$, $d_1(y):=\pi_1(\pi_2^{-1}(y)\cap E^1)\subset e_1$, is such that $\vert d_1(y)\vert=b-a$.

By taking the intersection of $e^2(x)$ with $c_2$, if necessary, we can assume that $e^2(x)\subset c_2$. This redefinition does not change the properties of $E^1$, which gets replaced as follows $E^1 \to E^1\cap (e_1\times c_2)$, but now $\pi_2(E^1)\subset c_2$, and for every $y\in c_2$, $d_1(y)=\pi_1(\pi_2^{-1}(y)\cap E^1)\subset e_1$, is such that $\vert d_1(y)\vert=b-a$.

Analogously, starting from Eq.\ (\ref{ni2}) and working with the roles of $x$ and $y$ exchanged we obtain that there is $c_1\subset e_1$, $\vert c_1\vert=b-a$, such that for every $y\in e_2$ there is $e^1(y)\subset c_1$, $\vert e^1(y)\vert = b-a$, such that on $E^2=\cup_{y\in e_2} e^1(y)\times\{y\}$  the derivative $\p_1 \p_2
f(x,y)$ exists and
\begin{equation} \label{bji}
\p_1 \p_2 f(x,y)=h(x,y).
\end{equation}
Moreover, $\pi_1(E^2)\subset c_1$ and for every $x\in c_1$,  $d_2(x):=\pi_2(\pi_1^{-1}(x)\cap E^2)\subset e_2$, is such that $\vert d_2(x)\vert=d-c$.

Let us define $E=E^1\cap E^2$, then $E\subset c_1\times c_2$, and for every $x\in c_1$,
\[
\pi_2(\pi_1^{-1}(x)\cap E)= \pi_2(\pi_1^{-1}(x)\cap E^1)\cap \pi_2(\pi_1^{-1}(x)\cap E^2)= e^2(x)\cap d_2(x),
\]
where both sets on the right-hand side have full measure thus $\vert  \pi_2(\pi_1^{-1}(x)\cap E)\vert=d-c$ (and analogously for the analogous statement with $x$ and $y$ exchanged). Finally, (\ref{bjy})-(\ref{bji}) are true on $E$, which proves (b) keeping (a1)-(a2) once we redefine $c_1\to e_1$, $c_2\to e_2$.
\end{proof}

We can also obtain a related theorem which adds information on the
differentiability properties of $f$:

\begin{theorem} \label{jui}
Let $f\colon [a,b]\times [c,d] \to \mathbb{R}$,  be such that
$f(x,\cdot):[c,d]\to \mathbb{R}$ and $f(\cdot,y): [a,b]\to
\mathbb{R}$ are absolutely continuous for every $x\in [a,b]$ and
$y\in [c,d]$, respectively. The following properties are equivalent:
\begin{itemize}
\item[(i)] There is  $e_1\subset (a,b)$, $\vert e_1\vert=b-a$, such that for $x\in e_1$, $\p_1
f(x,\cdot)$ exists  for every $y$, moreover it is absolutely
continuous over $[c,d]$, and $\p_2 \p_1 f\in L^1([a,b]\times
[c,d])$.
\item[(ii)]  There is  $e_2\subset (c,d)$, $\vert e_2\vert=d-c$, such that for $y\in e_2$, $\p_2 f(\cdot, y)$ exists for every $x$, moreover it  is absolutely continuous over
$[a,b]$, and $\p_1 \p_2 f\in L^1([a,b]\times [c,d])$.
\end{itemize}
Suppose they hold true then there is a subset $E\subset
e_1\times e_2$, $\vert E\vert=(b-a)(d-c)$, such that on $E$
the function $f$ is differentiable, $\p_2 \p_1 f(x,y)$, $\p_1 \p_2
f(x,y)$ exist, and
\[\p_2 \p_1 f=\p_1 \p_2 f.\]
\end{theorem}

\begin{proof}
Assume (i). For $x\in e_1$ function $\p_1 f(x,\cdot)$ is absolutely
continuous thus for every $y$
\begin{equation} \label{kid}
\p_1 f(x,y)-\p_1 f(x,c)=\int_c^y \p_2 \p_1 f(x,v) \dd v.
\end{equation}
Since $f(\cdot,y)$ is absolutely continuous we obtain upon
integration
\[
f(x,y)-f(a,y)-f(x,c)+f(a,c)=\int_a^x \dd u \int_c^y \p_2 \p_1 f(u,v)
\dd v.
\]
By Theorem \ref{the} applied to the right-hand side there is a
subset $\tilde{e}_1\subset (a,b)$, $\vert \tilde{e}_1\vert=b-a$,
such that Eq.\ (\ref{kid}) holds. This is already known to be true
with $\tilde{e}_1=e_1$. The same theorem establishes the existence
of $e_2\subset(c,d)$, $\vert e_2\vert=d-c$, such that for $y\in e_2$
and  for every $x\in (a,b)$
\[
\p_2 f(x,y)-\p_2 f(a,y)=\int_a^x \p_2 \p_1 f(u,y) \dd u.
\]
This last equation shows that for $y\in e_2$ the function $\p_2
f(\cdot,y)$ is absolutely continuous, thus for every $y\in e_2$, $\vert e_2\vert=d-c$, there is $e_1(y)\subset (a,b)$, $\vert e_1(y)\vert=b-a$, such that for $x\in e_1(y)$, and hence for almost every $(x,y)\in (a,b)\times (c,d)$  we have $\p_1\p_2 f(x,y)=\p_2\p_1 f(x,y)$ which implies that $\p_1\p_2 f \in L^1([a,b]\times [c,d])$, that is,  (ii) is true. The
proof that (ii) implies (i) is analogous.

It remains only to prove the differentiability of $f$ on $E$, the
remaining part of the last statement being an immediate consequence
of Theorem \ref{the}.

Let us prove the differentiability of $f$ at $(x_0,y_0)\in E$. The
partial derivative $\p_2 f(x, y_0)$ exists for every $x$ and is
absolutely continuous in $x$. Analogously, $\p_1 f(x_0, y)$ exists
for every $y$ and is absolutely continuous in $y$. We have
\begin{align*}
f(x_0+\Delta x,y_0+\Delta y)-f(x_0,y_0)&=[f(x_0+\Delta x,y_0+\Delta y)-f(x_0+\Delta x,y_0)]\\
& \quad +[f(x_0+\Delta x,y_0)-f(x_0,y_0) ],\\
& = \p_2 f(x_0+\Delta x,y_0) \Delta y+o_2(\Delta y)\\&\quad +\p_1 f(x_0,y_0) \Delta x+o_1(\Delta x)\\
& = [\p_2 f(x_0,y_0) +\p_1\p_2 f(x_0,y_0) \Delta x+o_3(\Delta x)]\Delta y\\
& \quad +o_2(\Delta y)+\p_1 f(x_0,y_0) \Delta x+o_1(\Delta x)\\
&=\p_2 f(x_0,y_0) \Delta y+\p_1 f(x_0,y_0) \Delta x +R(\Delta
x,\Delta y) ,
\end{align*}
where $R(\Delta x,\Delta y)/(\Delta x^2+\Delta y^2)^{1/2} \to 0$ for
the denominator going to zero.
\end{proof}

\begin{remark}
It is well known  that in the theory of distributions the equality
of mixed derivatives holds at any order of differentiation
\cite{vladimirov02}. In order to convert this fact into a claim for
ordinary differentiation it is necessary that the second derivatives
$\p_2 \p_1 f$ and $\p_1 \p_2 f$  be {\em regular} distributions,
namely representable as the integral of the test function $\varphi$
with $L^1(\bar{O},\mathbb{R})$ functions. Tolstov's Lemma allows us
to remove this double condition on the second mixed derivatives, for
it is sufficient to place that condition on just $\p_2\p_1 f$.
\end{remark}

\section{Lipschitz conditions on the partial derivatives}

Let us recall that a function $g\colon U \to \mathbb{R}^k$ defined
on an open set $U\subset \mathbb{R}^n$ is Lipschitz if for every
$p,q\in U$,
\[
\Vert g(p)- g(q)\Vert< K \Vert p-q\Vert
\]
for some $K>0$. It is locally Lipschitz if this inequality holds
over every compact subset of $U$, with $K$ dependent on the compact
subset.

A function $f\colon U\to \mathbb{R}$, $U\subset \mathbb{R}^2$,
$(x,y) \mapsto f(x,y)$, is differentiable with Lipschitz
differential, or $C^{1,1}$ for short, if $df\colon U\to
\mathbb{R}^2$ is Lipschitz. Clearly, if $f\in C^{1,1}$ with
Lipschitz constant $K$ then the partial derivative $\p_1 f(x,\cdot)$
regarded as a function of $y$ is $K$-Lipschitz. In particular, the
Lipschitz constant does not change if we change $x$. We say that
$\p_1 f(x,y)$  is Lipschitz in $y$ uniformly in $x$. Analogously,
$\p_2 f(x,y)$  is Lipschitz in $x$ uniformly in $y$, and two other
similar combinations hold.

For functions admitting Lipschitz partial derivatives the $L^1$
condition on the mixed derivative which we met in Theorem \ref{jui}
is satisfied.

\begin{theorem} \label{jun}
Let $f\colon [a,b]\times [c,d] \to \mathbb{R}$,  be such that
$f(x,\cdot):[c,d]\to \mathbb{R}$ and $f(\cdot,y): [a,b]\to
\mathbb{R}$ are absolutely continuous for every $x\in [a,b]$ and
$y\in [c,d]$, respectively. The following properties are equivalent:
\begin{itemize}
\item[(i)] There is  $e_1\subset (a,b)$, $\vert e_1\vert=b-a$, such that for $x\in e_1$, $\p_1
f(x,\cdot)$ exists for every $y$, and moreover it is Lipschitz over
$[c,d]$ uniformly for $x\in e_1$.
\item[(ii)]  There is  $e_2\subset (c,d)$, $\vert e_2\vert=d-c$, such that for $y\in e_2$, $\p_2 f(\cdot, y)$ exists for every
$x$,  and moreover it is  Lipschitz over $[a,b]$ uniformly for $y\in
e_2$.
\end{itemize}
Suppose they hold true then there is a subset $E\subset
e_1\times e_2$, $\vert E\vert=(b-a)(d-c)$, such that on $E$
the function $f$ is differentiable, $\p_2 \p_1 f(x,y)$, $\p_1 \p_2
f(x,y)$ exist and are bounded  and
\[\p_2 \p_1 f=\p_1 \p_2 f.\]
\end{theorem}

\begin{proof}
Suppose (i) is true and let $x\in e_1$ so that $\p_1 f(x,\cdot)$
exists and is $K$-Lipschitz over $[c,d]$. Then $\vert \p_2 \p_1 f
(x,\cdot)\vert\le K$ a.e.\ in $(c,d)$ and by the assumption of Lipschitz
uniformity this bound holds for  every $x\in e_1$, in particular $\vert \p_2 \p_1 f
\vert\le K$ holds almost everywhere on $O$. Thus $\p_2 \p_1
f\in L^{\infty}(\bar{O},\mathbb{R})\subset
L^{1}(\bar{O},\mathbb{R})$ and condition (i) of Theorem \ref{jui} is satisfied. In particular, the last statement of that theorem implies that  ``$\vert \p_2 \p_1 f
\vert\le K$, $f$ is differentiable and $\p_2 \p_1 f$, $\p_1 \p_2 f$ exist and coincide''  almost everywhere
on a subset $W\subset O$, $\vert W\vert=(b-a)(d-c)$.

 By Theorem \ref{jui} there is also  a
set $e_2$ such that if $y\in e_2$ then $\p_2 f(\cdot,y)$ exists for
every $x$, and moreover it is absolutely continuous thus
\[
\p_2 f(x,y)-\p_2 f(a,y)=\int_a^x \p_1 \p_2 f(u,y) \dd u.
\]
However, by Fubini's theorem there is $b_2\subset(c,d)$, $\vert b_2\vert=d-c$, such that for every $y\in b_2$, $\vert\pi^{-1}(y)\cap W\vert=b-a$. Thus for almost every $y$, namely for $y\in e_2\cap b_2$, we have that
$\p_2 f(\cdot,y)$ exists for
every $x$ and  it is absolutely continuous, and for almost every $x$ we have ``$\p_1\p_2 f=\p_2\p_1f$ and $\vert \p_2 \p_1 f
\vert\le K$'', thus for $y\in e_2\cap b_2$, $\p_2 f(\cdot, y)$ is $K$ Lipschitz over $(a,b)$ where $K$ does
not depend on $y$. Thus (ii) is proved once we rename $e_2\cap b_2$ as $e_2$.

%
%

The last statement follows from the previous paragraph or  from the last one in Theorem
\ref{jui}.
\end{proof}

For $C^1$ functions we have:

\begin{theorem} \label{thr}
Let $\Omega$ be an open subset of $\mathbb{R}^2$ and let  $f \in
C^1(\Omega,\mathbb{R})$ then the following conditions are
equivalent:
\begin{itemize}
\item[(i)] For every $x$, the partial derivative $\p_1 f(x,\cdot)$ is locally Lipschitz,
locally uniformly with respect to $x$.
 \item[(ii)] For every $y$, the partial derivative  $\p_2 f(\cdot,y)$ is locally Lipschitz,
locally uniformly with respect to $y$.
\end{itemize}
If they hold true, for instance $f \in
C^{1,1}_{loc}(\Omega,\mathbb{R})$, then on a set $E\subset \Omega$, $\vert \Omega \backslash E\vert=0$,
 $\p_2 \p_1 f$ and $\p_1
\p_2 f$ exist, $f$ is differentiable and  $\p_2 \p_1 f=\p_1 \p_2 f$. In particular,  $\p_2 \p_1 f$ and $\p_1
\p_2 f$ belong to
$L^\infty_{loc}(\Omega,\mathbb{R})$.
\end{theorem}

\begin{proof}
Let $p\in \Omega$ and let us consider an open neighborhood
$(a,b)\times (c,d)$ of $p$ such that $\bar{O}\subset \Omega$. Let us
assume (i). By Theorem \ref{jun} we need only to show that  $\p_2
f(\cdot,y)$ is $K$-Lipschitz in $(a,b)$ for every chosen value of
$y\in (c,d)$ provided it is so for almost every $y\in (c,d)$. Let
$y\in (c,d)$ and let $\epsilon>0$. The function $\p_2 f$ is
continuous, thus uniformly continuous over the compact set
$[a,b]\times [c,d]$. We can find a $\delta>0$ such that whenever
$\vert y_2-y_1\vert<\delta$, $y_1,y_2\in [c,d]$, we have $\vert \p_2
f(x,y_2)-\p_2 f(x,y_1)\vert<\epsilon$ for every $x\in [a,b]$. We can
find some $\bar{y}\in (y-\delta,y+\delta)\cap [c,d]$ such that $\p_2
f(\cdot, \bar{y})$ is $K$-Lipschitz. Thus
\begin{align*}
\vert \p_2 f(x_2,y)-\p_2 f(x_1,y)\vert &\le \vert \p_2 f(x_2,\bar
y)-\p_2 f(x_1,\bar y) \vert+\vert \p_2 f(x_2,\bar y)-\p_2 f(x_2,
y)\vert
\\& \ +\vert \p_2 f(x_1,\bar y)-\p_2 f(x_1, y)\vert\le K \Vert
x_2-x_1\Vert+2 \epsilon
\end{align*}
From the arbitrariness of $\epsilon$, $x_1$ and $x_2$ we obtain that
$\p_2 f(\cdot,y)$ is $K$-Lipschitz for every chosen value of $y$.
The remaining claims follow trivially from  Theorem \ref{jun}.
\end{proof}


We stress that if $f\in C^1$ then the fact that $\p_1 f(x,y)$ is
Lipschitz in $y$ uniformly in $x$, and that $\p_2 f(x,y)$ is
Lipschitz in $x$ uniformly in $y$, does not guarantee that $f\in
C^{1,1}$; it is sufficient to consider the function $f(x,y)=\vert
x\vert^{3/2}$. As a consequence, the assumptions of this theorem are
 weaker than $f\in C^{1,1}_{loc}(\Omega,\mathbb{R})$.

%
%

For the $C^{1,1}_{loc}(\Omega,\mathbb{R})$ case the equality of
mixed derivatives can also be obtained as a consequence of Young's
(point 3 above) and Rademacher's theorems. We recall that the latter
states that every Lipschitz function is almost everywhere
differentiable \cite{evans92}. Indeed:

\begin{theorem} \label{nde}
Let $f\colon \Omega \to \mathbb{R}$, $f \in C^{1,1}_{loc}$, then $f$
is twice differentiable almost everywhere and in such
differentiability set $\p_2 \p_1 f=\p_1\p_2 f$.
\end{theorem}

\begin{proof}
The differential $d f\colon \Omega\to \mathbb{R}^2$ is Lipschitz
thus differentiable almost everywhere (Rademacher's theorem). If $p$
belongs to the differentiability set then $\p_1 f$ and $\p_2 f$,
being components of the differential, are there differentiable. Thus
by  Young's theorem 3.\ we have $\p_2 \p_1 f=\p_1\p_2 f$ at $p$.
\end{proof}

\section{Some applications}

In this section we explore some applications that motivated our
study. They are in the area of differential geometry  but it is
likely that many other applications can be found.

\subsection{Usefulness of Lipschitz one-forms}

A rather natural application of these results is in the study of
Lipschitz 1-forms, namely 1-forms with Lipschitz coefficients, over
differentiable manifolds (at least $C^{1,1}$). Indeed, some related
results have been already developed following paths independent of
the above considerations.

 If $f\in
C^{1,1}_{loc}$ then by Theorem \ref{nde} $\dd^2 f=0$ almost
everywhere in the Lebesgue 2-dimensional measure of any
2-dimensional $C^{1,1}$ embedded manifold. Thus the exterior
differential satisfies $\dd^2=0$ in a well defined sense whenever
 0-forms and 1-forms
are taken with the correct degree of differentiability. In
particular, if $\omega$ is a Lipschitz 1-form then Stokes theorem
\[
\int_S \dd \omega=\int_{\p S} \omega
\]
still holds true \cite{simic96}.

One also expects that Lipschitz distributions of hyperplanes should
be integrable according to the usual rule for $C^1$ distributions.
Namely, let $\omega$ be a Lipschitz 1-form, then the distribution
$\textrm{Ker}\,\omega$ should be integrable if and only  if
$\omega\wedge \dd \omega=0$. This result has indeed been proved
\cite{simic96,rampazzo07}.

The nice behavior of locally Lipschitz 1-forms suggests to study
\mbox{(pseudo-)} Riemannian $C^{2,1}$ manifolds endowed with  Lipschitz
connections $\nabla$ and $C^{1,1}$ metrics. Indeed, in Cartan's
approach the connection is regarded as a Lie algebra-valued 1-form
in the bundle of reference frames. In such framework the Riemannian
tensor would be a locally bounded element of $L_{loc}^\infty(M)$ and
hence would be defined only almost everywhere in the Lebesgue
2-dimensional measure. In particular, it could be discontinuous
though locally bounded. This feature would be quite appreciated in
the theory of Einstein's general relativity. There the Ricci tensor
is proportional to the stress-energy tensor and so it is necessarily
discontinuous for the typical mass distribution of a planet; the
reader may consider the discontinuity in density which takes place
at the planet's boundary.

\subsection{An application to mathematical relativity}

As mentioned, the assumptions  of Theorem \ref{thr} are weaker than
the condition $f\in C^{1,1}_{loc}$. We wish to describe shortly  an
example of application where those weaker conditions turn out to be
important.

 In Einstein's General Relativity the spacetime continuum
is represented with a Lorentzian manifold \cite{hawking73}, namely a
differentiable manifold endowed with a metric of signature
$(-,+,+,+)$. Observers or  massive particles are represented by
$C^1$ curves $x(s)$   which are timelike, namely such that $
g({x}',{x}')<0$. Unfortunately, for various mathematical arguments
it is necessary to consider limits of such curves, and those limits
are rarely $C^1$ but are necessarily Lipschitz. Lipschitz
mathematical objects arise quite naturally in General Relativity,
ultimately because the light cones on spacetime (at $p\in M$ the
light cone is given by the subset of $T_pM$ where $g$ vanishes)
place a  bound on the local speed of massive objects.

%
%
%
%
%
%
%
%
%
%
%
%

A typical problem which is met in the discussion of the clock effect
(or twin paradox)  deals with curve variations $x(t,s)$ of timelike
geodesics $x(\cdot, s)$ parametrized by $s$, where the transverse
curves $x(t, \cdot)$ are just Lipschitz. In this situation the
properties of the exponential map allow one to prove that the
tangent $\p_t x(t,s)$ is Lipschitz in $s$ uniformly in $t$, exactly
the assumptions of Theorem \ref{thr} (the function need not be
differentiable in $s$). It turns out that in order to prove formulas
such as the first variation formula for the energy functional  of
differential geometry,
\[
E[x]=\frac{1}{2} \int_0^1 g(\dot x, \dot x) \dd t,
\]
one needs to switch $\p_s \p_t x$ for $\p_t \p_s x$ an operation
which is indeed allowed thanks to Theorem \ref{thr}. Thus, this
theorem can be used to operate with Lipschitz curves in Lorentzian
(or Riemannian) geometry much in the same way as it is usually done
with $C^1$ curves \cite{minguzzi13d}.


\section{Conclusions}


We have reviewed the notion of strong differentiation and Mikusi\'nski's result on the equality of mixed partial derivatives.
The assumptions do  not demand the existence and
continuity of any second derivative in a neighborhood of the point.
Rather, the theorem assumes the weaker notion of strong differentiability of
one first derivative at the point. This possibility was suggested by
previous applications of the concept of strong differentiation where
it proved to be particularly advantageous, e.g.\ the inverse
function theorem.

We have then considered results which prove the existence and
equality of mixed partial derivatives almost everywhere. We have
presented and elaborated previous results by Tolstov  stressing the
importance of the Lipschitz condition on first partial derivatives
for applications. The advantage of this approach over alternative distributional approaches becomes clear
whenever one cannot conclude that both
mixed second partial derivatives are summable. We have ended this
work giving a specific example of application where this classical
approach is more effective and justified.

\medskip
\noindent {\bf Acknowledgment}. I thank an anonymous referee for some useful criticisms. Work partially supported by GNFM of INDAM.

\end{document}